\documentclass[oneside,reqno]{amsart}
\usepackage{hyperref,amssymb,amsmath,amsthm,color}
\usepackage[all]{xy}
\usepackage{geometry,mdframed}
\hypersetup{
	colorlinks = true,
	linkcolor = blue,
	citecolor = blue
}
\usepackage{tikz-cd}
\definecolor{orangetwo}{rgb}{0.9803922, 0.4196078, 0.3058824}
\definecolor{freading}{rgb}{ 0.8980392,1.0000000,0.9058824}
\definecolor{gray}{rgb}{0.25,0.25,0.25}
\definecolor{dgray}{rgb}{0.15,0.15,0.15}
\definecolor{white}{rgb}{1,1,1}

\definecolor{theoremback}{rgb}{0.8235294, 0.8627451, 0.9137255}
\definecolor{exampletitle}{rgb}{0.6705882, 0.7568627, 0.8705882}

\usepackage{fancyhdr}
\newtheoremstyle{nthm}
{1pt}
{1pt}
{\itshape}
{0em}
{\bfseries}
{.}
{.5em}
{}

\newtheoremstyle{ndef}
{3pt}
{3pt}
{}
{0em}
{\bfseries}
{.}
{.5em}
{}

\newtheoremstyle{nrem}
{3pt}
{3pt}
{}
{0em}
{\itshape}
{.}
{.5em}
{}

\swapnumbers
\theoremstyle{nthm}

\newmdtheoremenv[linecolor=theoremback,linewidth=2,backgroundcolor=theoremback,skipabove=6,skipbelow=6]{thm}[subsection]{Theorem}

\newtheorem{prop}[subsection]{Proposition}

\newtheorem{conj}[subsection]{Conjecture}
\theoremstyle{ndef}

\theoremstyle{nrem}

\newtheorem{rem}[subsection]{Remark}
\newtheorem{exmp}[subsection]{Example}

\newcommand{\F}{\mathbb{F}}
\newcommand{\G}{\mathbb{G}}

\newcommand{\Z}{\mathbb{Z}}

\newcommand{\UU}{\mathrm{U}}

\newcommand{\afr}{\mathfrak{a}}

\newcommand{\gfr}{\mathfrak{g}}

\newcommand{\ofr}{\mathfrak{o}}

\newcommand{\ufr}{\mathfrak{u}}
\newcommand{\vfr}{\mathfrak{v}}

\newcommand{\Fcl}{\mathcal{F}}

    \DeclareMathOperator{\Gal}{Gal}

	\newcommand{\sets}{\mathbf{Set}}

	\newcommand{\schemes}{\mathbf{Sch}}

	\DeclareMathOperator{\Hom}{Hom}
	\DeclareMathOperator{\Spec}{Spec}

	\newcommand{\GL}{\mathrm{GL}}

\DeclareMathOperator{\Ad}{Ad}

\newcommand{\op}{\mathrm{op}}

\DeclareMathOperator{\Gr}{Gr}

\newcommand{\df}[1]{\textcolor{orangetwo}{{\bfseries #1}}\index{#1}}

\newsavebox{\fmbox}

\DeclareMathOperator{\diag}{diag}

\begin{document}
\title[Affine Springer Fibers for Symmetric Spaces]{ {\bfseries Dimensions of Affine Springer Fibers for some $\GL_n$ Symmetric Spaces} }
\author{Jason K.C. Pol\'ak}
\begin{abstract}
	We determine a formula for the dimension of a family of affine Springer fibers associated to a symmetric space arising from the block diagonal embedding $\GL_n\times\GL_n\to\GL_{2n}$. As an application, we determine the dimension of affine Springer fibers attached to certain unitary symmetric spaces.
\end{abstract}
\date{\today}
\maketitle

\thispagestyle{plain}
\tableofcontents

\section{Introduction}
Let $G$ be a connected reductive algebraic group over a finite field $k$. Let $F= k( (t))$ be the Laurent series field over $k$ and let $\ofr = k[ [t]]$ be the ring of integers of $F$. Write $v:F^\times\to \Z$ for the valuation on $F$ corresponding to the uniformiser $t$. For any representation $\rho$ of $G$ on a vector space $V$ defined over $\ofr$, and $\gamma\in V(F)$ the set\begin{align*}
	X(G,\gamma)(k) = \{ g\in G(F)/G(\ofr) : \rho(g)^{-1}\gamma\in V(\ofr)\},
\end{align*}
is the set of $k$-points of an ind-scheme $X(G,\gamma)$ over $k$ called the affine Springer fiber with respect to $\gamma$. We have left $V$ out of the notation since the representation shall be clear from context. Finding a formula for the dimension of $X(G,\gamma)$ when its dimension is actually finite is an important and intriguing problem in many parts of representation theory. When $\rho$ is the adjoint representation, this was done by Kazhdan and Lusztig for $G$ split over $F$ with $\gamma$ regular and in the Lie algebra of a split maximal torus~\cite{Kazhdan1988}, and by Bezrukavnikov for any connected reductive $G$ and any regular semisimple $\gamma$ \cite{Bezrukavnikov1996}.

Let $G = \GL_{2n}$ and consider the involution $\theta:G\to G$ defined by
\begin{align*}\theta(x) = 
	\begin{pmatrix}
		I_n & 0\\
		0 & -I_n
	\end{pmatrix}x\begin{pmatrix}
		I_n & 0\\
		0 & -I_n
	\end{pmatrix}.
\end{align*}
The fixed-point group $G_0 := G^\theta\cong \GL_n\times\GL_n$ sits inside $\GL_{2n}$ block-diagonally. Write $\gfr$ for the Lie algebra of $G$. We consider the representation of $G_0$ on $\gfr_1 := \{ x\in \gfr : \theta(x) = -x\}$. Let $\gamma\in\gfr_1(F)$ be semisimple, and also \emph{regular}, which means that the orbit $G_0\gamma$ has maximal dimension amongst $G_0$-orbits in $\gfr_1$. In this paper, we compute the dimension of the affine Springer fibers $X(G_0,\gamma)$ when $\gamma$ is regular and semismple, analogous to the formula proved in \S 5 of \cite{Kazhdan1988}. More precisely:
\begin{thm}
	Let $\gamma = \left(\begin{smallmatrix} 0 & X\\ Y & 0\end{smallmatrix}\right) \in\gfr_1(F)$ be a regular semisimple element conjugate to $\left(\begin{smallmatrix}0 & 1\\ \beta & 0\end{smallmatrix}\right)\in\gfr_1(F)$ for some $\beta = \diag(x_1,\dots,x_n)$ with $x_i\not= x_j$ whenever $i\not=j$. Then
	\begin{align*}
		\dim_k X(G_0,\gamma) = v(\det([XY,-]) =\sum_{i < j} v(x_i - x_j).
	\end{align*}
    Here, $XY$ is considered as an element of $\mathfrak{gl}_n$ and $[XY,-]$ is the commutator map restricted to the Lie algebra of the upper triangular Borel subgroup.
\end{thm}
To prove this, we may as well assume that $\gamma = \left(\begin{smallmatrix} 0 & 1\\ \beta & 0\end{smallmatrix}\right)$. We then proceed in two steps: first, we show that there is a well-defined map $X(G_0,\gamma)\to X(T,\gamma)$ where $T$ is the diagonal maximal torus, and use the idea in Kazhdan-Lusztig~\cite{Kazhdan1988} to determine the dimensions of the fibers over each $t\in X(T,\gamma)$. Second, we show that the dimension of the fibers is independent of $t$, which is not immediately apparent in our case but was obvious in Kazhdan-Lusztig.

Although affine Springer fibers are interesting in their own right and have been studied for some time now, we are also motivated by the application of our formula to unitary groups as in Example~\ref{exmp:unitary}, and ultimately by the relative trace formula in~\cite{getzw2013}. For these unitary group and other symmetric space representations as studied in~\cite{Kostant1971, levyInvolutions}, we hope to derive a fundamental lemma analogous to the adjoint fundamental lemma proved in~\cite{ngo2010lemme}. Roughly speaking, the fundamental lemma is a formula relating the number of points on affine Springer fibers of one group to the number of points on the affine Springer fibers of another \emph{endoscopic} group. A key part of this formula is the mysterious \emph{transfer factor} that should be related to the dimension of the affine Springer fibers. For the adjoint case, this relationship was worked out in~\cite{ngo2010lemme}, but for symmetric spaces, although an example of a transfer factor was determined in~\cite{Polak2015}, the precise relationship between the dimension of affine Springer fibers and the transfer factors remains elusive.

\subsection{Notations and Conventions} We assume that all reductive groups are connected. We write $\schemes$ for the category of schemes and $\sets$ for the category of sets. If $a_1,\dots,a_n$ are elements of a ring then $\diag(a_1,\dots,a_n)$ denotes the corresponding $n\times n$ diagonal matrix. Any element of $\gfr_1$ is of the form
\begin{align*}
    \begin{pmatrix}
        0 & X\\
        Y & 0
    \end{pmatrix}
\end{align*}
where $X$ and $Y$ are $n\times n$ matrices; for brevity we write this element as $(X,Y)$.

\subsection*{Acknowledgements} The author thanks Jayce R. Getz for discussions and suggesting numerous improvements.

\section{Ind-Schemes}

For any directed system $Y_1\hookrightarrow Y_2\hookrightarrow Y_3\hookrightarrow\cdots$ of schemes where each morphism $Y_i\hookrightarrow Y_{i+1}$ is a closed embedding, we define a functor $Y = \varinjlim Y_i:\schemes^\op\to\sets$ by first setting $Y(S) = \varinjlim \Hom_{\schemes}(S,Y_i)$ for every affine scheme $S$, and then extend $Y$ to a contravariant functor on all schemes by taking its Zariski sheafification. We shall call any functor naturally equivalent to one of this form an \df{ind-scheme}. The dimension of $Y$ is by definition $\sup_i\dim Y_i$, and is independent of its presentation as a directed system. 

From now on we consider schemes over a finite field $k = \F_q$ of odd characteristic. We write $F = k( (t))$ and $\ofr = k[ [t]]$. For us, the prototypical example of an ind-scheme over $k$ is the affine Grassmanian whose set of $k$-points is $G(F)/G(\ofr)$ where $G$ is a reductive algebraic group over $k$. As a functor the affine Grassmanian for $G$, denoted by $\Gr(G)$, is usually defined to be the fpqc-sheafification of the functor defined for each $k$-algebra $R$ by
\begin{align*}
	R \longmapsto G\Bigl(R( (t))\Bigr)/G\Bigl(R[ [t]]\Bigr).
\end{align*}
One can recast this definition into more concrete language:
\begin{prop}[\cite{Richarz2014}, Lemma 1.1]
	The affine Grassmannian represents the functor that assigns to every $k$-algebra $R$ the set of isomorphism classes of pairs $(\Fcl,\varphi)$ where $\Fcl$ is a $G$-torsor over $\Spec(R[ [t]])$ and $\varphi$ is a trivialisation of $\Fcl[t^{-1}]$ over $\Spec(R( (t))$
\end{prop}
When $G = \GL_n$, a $G$-torsor over $\Spec(R[ [t]])$ is the same thing as an algebraic vector bundle of rank~$n$ over $\Spec(R[ [t]])$, which is in turn the same thing as $R[ [t]]$-projective module of rank~$n$. So the $k$-points of the affine Grassmannian for $\GL_n$ can be described as the set isomorphisms $\varphi:\ofr^n\otimes_{\ofr}F \to F^n$ under the following equivalence relation: $\varphi\simeq \varphi'$ if and only if there exists an isomorphism $\beta:\ofr^n\to \ofr^n$ making the diagram
\begin{equation*}
    \SelectTips{cm}{}
	\xymatrix{
    \ofr^n\otimes_{\ofr}F^n\ar[rd]^\varphi\ar[dd]_{\beta\otimes 1} & {}\\
        {}  &F^n\\
        \ofr^n\otimes_{\ofr}F^n\ar[ru]_{\varphi'} & {}
        }
\end{equation*}
commute. In other words, $\Gr(\GL_n)(k)\cong \GL_n(F)/\GL_n(\ofr)$. The following Proposition~\ref{prop:descgkl} is a description of the ind-scheme structure on $\Gr(\GL_n)$ due to Gaitsgory~\cite{Gaitsgory2000}, generalising the ind-scheme structure for semisimple simply connected groups due to Kazhdan and Lusztig.
\begin{prop}\label{prop:descgkl}
	On $k$-schemes, the functor $\Gr(\GL_n)$ is naturally equivalent to the Zariski sheafification of the direct limit $\Gr^1\hookrightarrow \Gr^2\hookrightarrow \cdots$ where $\Gr^m(R)$ is the set of all $R$-flat $t$-stable submodules of $R\otimes_k t^{-m}V[ [t]]/t^{m+1}V[ [t]]$.
\end{prop}
One can give an ind-scheme structure on the affine Grassmannian for all reductive groups $G$ by showing that a closed embedding $G\hookrightarrow \GL_n$ induces a closed embedding $\Gr(G)\to \Gr(\GL_n)$ of affine Grassmannians, and in fact one can even drop the reductive hypothesis (see~\cite{Gaitsgory2000} for details). Let us now take a look at Example~\ref{exmp:zerodimtorus}, which is the starting point for our investigations into the dimensions of affine Springer fibers.

\begin{exmp}[Dimension for Tori]\label{exmp:zerodimtorus}To get an intuition for the geometric structure of the affine Grassmanian, consider $G = \G_m$. Then $G(F)/G(\ofr) = F^\times/\ofr^\times$ which is isomorphic to $\Z$ as a group. As described in Proposition~\ref{prop:descgkl}, the functor $\Gr(\G_m)$ can be written as a colimit of schemes $Y_1\hookrightarrow Y_2\hookrightarrow\cdots$ where for each $m$ and for any field extension $k'/k$, the set $Y_m(k')$ is the set of all $t$-stable $k'$-subspaces of $t^{-m}k'( (t))/t^{m+1}k'( (t))$. This space is just isomorphic to $k'^{2m + 2}$ where multiplication by $t$ is the \emph{right-shift} linear operator $(a_1,a_2,\dots,a_{2m+2})\mapsto (0,a_1,\dots,a_{2m+1})$. For a fixed integer $0\leq d\leq 2m + 2$, there is a unique $d$-dimensional subspace that is $t$-invariant.
    
    Hence $Y_m(k')$ consists of one point in each irreducible component of the disjoint union over the $2m + 2$ Grassmannians for $t^{-m}k'( (t))/t^{m+1}k'( (t))$ and so $\Gr(\G_m)$ is zero-dimensional. The same argument also shows that $\Gr(T)$ is zero-dimensional whenever $T\cong \G_m^n$.
\end{exmp}

\section{The Fibration}
The formula of Kazhdan and Lusztig was proved by first considering elements in the Lie algebra of a split maximal torus to obtain an explicit formula. The analog here is the following type of element: let $\gamma\in\gfr_1(F)$ be a regular semisimple element of the form
\begin{align}\label{eqn:gammaform}
	\gamma = 
	\begin{pmatrix}
		0 & I_n\\
        \beta & 0
	\end{pmatrix}
\end{align}
where $\beta = \diag(x_1,\dots,x_n)$ with $x_i\not= x_j$ for $i\not= j$ and $I_n$ is the $n\times n$ identity (cf. \cite[Proposition~2.1]{Jacquet1996}). Incidentally, the choice of such a $\gamma$ determines a Cartan subspace, the analogue of a Cartan subalgebra for symmetric spaces, and the formula we will derive gives the dimension of the corresponding affine Springer fiber for any element in such a subspace containing $\gamma$; see Remark~\ref{rem:cartansub} for further details.

We let $T\subset G_0$ be the diagonal maximal split torus and $B\subset G_0$ be the Borel subgroup of upper triangular matrices in each $n\times n$ block. Then we have an Iwasawa decomposition $G_0(F) = T(F)U(F)G_0(\ofr)$ where $U$ is the unipotent radical of $B$.  
\begin{prop}For each $g\in G_0(F)/G_0(\ofr)$, fix a decomposition $g = t_gu_g\in G_0(F)/G_0(\ofr)$ with $t_g\in T(F)$ and $u_g\in U(F)$. Then the map
    \begin{align*}
        p:G_0(F)/G_0(\ofr)&\longrightarrow T(F)/T(\ofr)\\
        g&\longmapsto t_g
    \end{align*}
    is well-defined.
\end{prop}
\begin{proof}
    It suffices to observe that if $t_1u_1 = t_2u_2\in G_0(\ofr)$ then $u_2^{-1}t_2^{-1}t_1u_1\in G_0(\ofr)$ and so $t_2^{-1}t_1\in G_0(\ofr)\cap T(F) = T(\ofr)$.
\end{proof}
Similarly, we can prove:
\begin{prop}
	If $g\in X(G_0,\gamma)$ then $t = p(g)\in X(T,\gamma)$.
\end{prop}
\begin{proof}
	Write $u = \diag(v,w)$ where $v,w\in\GL_n(F)$ and similarly for $t = \diag(r,s)$ for $r,s\in \GL_n(F)$. Let $g = tu\in X(G_0,\gamma)$, so that $\Ad(g)^{-1}\gamma = \Ad(u)^{-1}\Ad(t)^{-1}\gamma\in \gfr_1(\ofr)$. This is the same as saying that the matrices
	\begin{align*}
		v^{-1}r^{-1}sw, w^{-1}s^{-1}\beta rv
	\end{align*}
	have integral entries. In particular, this implies that $r^{-1}s$ and $s^{-1}\beta r$ have integral entries since $\beta$ is also a diagonal matrix.
\end{proof}

\section{The Dimension of the Fiber}

We wish to determine the dimension of $X(G_0,\gamma)$. In this section, we determine the dimension of a fiber $p^{-1}(t)$, and show that this dimension is independent of $t$. From Example~\ref{exmp:zerodimtorus} shows that the dimension of $X(T,\gamma)$ is zero, and so this is sufficient to calculate the dimension of $X(G_0,\gamma)$.

First, let us introduce some notation. Since $G_0\cong \GL_n\times\GL_n$, for a matrix $A\in G_0(F)$ we use the notation \emph{$(i,j)$ coordinate} to refer to the $(i,j)$ coordinate (in the usual sense) in the first block, and we use \emph{$(i,j)'$ coordinate} for the $(i,j)$ coordinate in the second block.

For positive integers $i$ and $j$, denote by $V_{i,j}$ a copy of the vector space $k$. For each $\ell\in \Z$ write
\begin{align*}
    V_\ell = \bigoplus_{\substack{j - i = \ell\\ 1\leq i,j\leq n} } V_{i,j}.
\end{align*}

Fix a $t\in T(F)$. The fiber $p^{-1}(t)$ is the affine Springer fiber
\begin{align*}
    X(U,\Ad(t)^{-1}\gamma) = \{ u\in U(F)/U(\ofr) : \Ad(u^{-1})\Ad(t)^{-1}\gamma\in \gfr_1(\ofr)\}.
\end{align*}
Define a map $p_1:U(F) \to V_1^2(F)$ by sending $u\in U(F)$ to the element in $V_1^2(F)$ whose $(i,j)$th entry in the first copy of $V_1$ is the $(i,j)$th entry of $\Ad(u)^{-1}\Ad(t)^{-1}\gamma$, and whose $(i,j)$th entry in the second copy of $V_1$ is the $(i,j)'th$ entry of $\Ad(u)^{-1}\Ad(t)^{-1}\gamma$.

\begin{prop}The map $p_1$ induces a well-defined surjective map
\begin{align*}
    q_1:Y_1:=X(U,\Ad(t)^{-1}\gamma)\longrightarrow V_1^2(\ofr)/p_1(U(\ofr)) =: Z_1
\end{align*}
and all the fibers of $q_1:Y_1\to Z_1$ have the same dimension.
\end{prop}
\begin{proof}
    That $q_1$ is well-defined is clear. Let us prove that the map is surjective. We look at the $(i,j)$th entry and the $(i,j)'$th entry of $\Ad(u)^{-1}\Ad(t)^{-1}\gamma$ for $j - i = 1$ and $1\leq i,j\leq n$; these are respectively
    \begin{equation}\label{eqn:mainform}
        \begin{split}
            f_{ij} &= b_{ij}s_ir_i^{-1} - a_{ij}s_jr_j^{-1},\\
            g_{ij} &= a_{ij}x_ir_is_i^{-1} - b_{ij}x_jr_js_j^{-1}.
    \end{split}
    \end{equation}
    These form the entries of the column vector
    \begin{align*}
        \begin{pmatrix}
            -s_jr_j^{-1} & s_ir_i^{-1}\\
            x_ir_is_i^{-1} & -x_jr_js_j^{-1}
        \end{pmatrix}
        \begin{pmatrix}
            a_{ij}\\
            b_{ij}
        \end{pmatrix}.
    \end{align*}
    To prove the surjectivity of $q_1$, it suffices to observe that the $2\times 2$ matrix appearing in this formula is invertible, since it has determinant $-(x_i - x_j)$, and $\gamma$ being regular is equivalent to $x_i\not= x_j$. Now let us show that the fibers have constant dimension. If we look at the $(i,j)$ and $(i,j)'$ coordinates for the elements on the diagonal $j - i=2$, these are the sum of~\eqref{eqn:mainform} with something depending only on the coordinates in the diagonal $j-i=1$. In other words, to use the notation as in~\eqref{eqn:mainform}, the nonintegral coordinates of $f_{ij}$ and $g_{ij}$ for $j-i = 2$ are uniquely determined once $z\in Z_1(k)$ is chosen. Repeating this argument, the coordinates for $j - i=3$ are a sum of those as in~\eqref{eqn:mainform} with something that depends only on the coordinates for $j-i=1,2$. This gives an isomorphism between the various fibers, showing that they have the same dimension.
\end{proof}

To calculate the dimension of $X(U,\Ad(t)^{-1}\gamma)$, we have to calculate the dimension of any fiber of $q_1$ and of $Z_1$. Write $Y_2 = q_1^{-1}(0)$ and define a map $q_2:Y_2\to V_2(\ofr)/p_2(U_2(\ofr)) =: Z_2$ analogously as above where $U_2$ is the unipotent subgroup of $U$ with each $i,j$-entry zero where $j - i = 1$. 

We recall the following basic theorem in algebraic geometry:
\begin{prop}[\cite{Goertz2010}, Corollary 14.119]\label{thm:agbasic}
    Let $X$ and $Y$ be $k$-schemes of finite type and $f:X\to Y$ be a dominant morphism of finite type. If all the nonempty fibers of $f$ have dimension $r$ then $\dim X = \dim Y + r$.\qedsymbol
\end{prop}
Applying Proposition~\ref{thm:agbasic} inductively, we obtain
\begin{align*}
    \dim X(U,\Ad(t)^{-1}\gamma) = \sum \dim Z_i.
\end{align*}
To calculate $\dim Z_i = \dim V_i(\ofr)/p_i(U_i(\ofr))$, we recall that $\Ad(u)^{-1}\Ad(t)^{-1}\gamma$ has the same $i,j$ entries for $j - i =0,1$ as
\begin{align*}
    \Ad(t)^{-1}\gamma + [u^{-1} - 1,\Ad(t)^{-1}\gamma].
\end{align*}
We put $-1$ in the commutator so that $u^{-1}-1$ is in the Lie algebra of $U$. We note now that the $i,j$-coordinate for the left-hand side for $j > i$ (i.e. above the diagonal) is the same as the $i,j$-coordinate of $[u^{-1} - 1,\Ad(t)^{-1}\gamma]$, since $\Ad(t)^{-1}\gamma$ is just the diagonal part of $\Ad(u)^{-1}\Ad(t)^{-1}\gamma$.

In other words, what we see is that $Z_i$ is the space $V_i(\ofr)$ modulo the image of $\ufr_i$ under the linear map $[-,\Ad(t)^{-1}\gamma]$. So, $\sum\dim Z_i$ is the dimension of the $k$-vector space $\oplus V_i/[\ufr(\ofr),\Ad(t)^{-1}\gamma]$. By reducing the square matrix of $[-,\Ad(t)^{-1}\gamma]$, or equivalently of $\varphi := [\Ad(t)^{-1}\gamma,-]$ to diagonal form via row and column operations, we see that $\sum \dim Z_i = v(\det\varphi)$.

We next determine a formula for $v(\det(\varphi))$. For any $1 \leq i,j\leq n$ denote by $e_{ij}$ the matrix whose only nonzero entry is~$1$ at $(i,j)$. Then an $F$-basis of $\ufr(F)$ is $\{ (e_{ij},\mathbf{0}) : i < j\} \cup \{ (\mathbf{0},e_{ij}) : i < j\}$ where $\mathbf{0}$ is the zero matrix. We also use this notation for a basis of $\vfr(F)$, where $\vfr\subset\gfr_1$ is the subspace consisting of pairs $(X,Y)\in\gfr_1$ with $X$ and $Y$ strictly upper triangular. For example, in $\ufr(F)$, the basis element $(e_{ij},\mathbf{0})$ corresponds to the block diagonal matrix $\diag(e_{ij},\mathbf{0})$, whereas used to denote a basis element of $\vfr(F)$, it corresponds to the block antidiagonal matrix
\begin{align*}
	\begin{pmatrix}
		\mathbf{0} & e_{ij}\\
		\mathbf{0} & \mathbf{0}
	\end{pmatrix}
\end{align*}
These bases are also bases for $\ufr(\ofr)$ and $\vfr(\ofr)$ as free $\ofr$-modules. We need to write down the $n(n-1)\times n(n-1)$ matrix of $\varphi$. Recall that we have fixed $t\in X(T,\gamma)$. Write $t = (r,s)$ where $r = \diag(r_1,\dots,r_n)$ and $s= (s_1,\dots,s_n)$. Then $\Ad(t)^{-1}\gamma$ is the matrix 
\begin{align*}
(\diag(r_1^{-1}s_1,\dots,r_n^{-1}s_n),\diag(s_1^{-1}r_1x_1,\dots,s_n^{-1}r_nx_n))\in\gfr_1(F).
\end{align*}
We compute the matrix with respect to the ordered basis 
\begin{align*}
    \{ (e_{11},\mathbf{0}),(e_{12},\mathbf{0}),\dots,(e_{n-1,n},\mathbf{0}),(\mathbf{0},e_{11}),\dots,(\mathbf{0},e_{n-1,n}) \}
\end{align*}
for $\ufr(F)$ and the same notation denotes the ordered basis we use for $\vfr(F)$. Then it is somewhat trivial to write down the matrix corresponding to $\varphi$:
\begin{equation}\label{eqn:largematrix}
\varphi = \begin{pmatrix}
  -a_2 & ~ & ~ & ~ & ~ & a_1 & ~ & ~ & ~ & ~\\
  ~ & \ddots & ~ & ~ & ~ & ~ & \ddots & ~ & ~ & ~\\
  ~ & ~ & -a_n & ~ & ~ & ~ & ~ & a_1 & ~ & ~\\
  ~ & ~ & ~ & \ddots & ~ & ~ & ~ & ~ & \ddots & ~\\
  ~ & ~ & ~ & ~ & -a_n & ~ & ~ & ~ & ~ & a_{n-1}\\
  b_1x_1 & ~ & ~ & ~ & ~ & -b_2x_2 & ~ & ~ & ~ & ~\\
  ~ & \ddots & ~ & ~ & ~ & ~ & \ddots & ~ & ~ & ~\\
  ~ & ~ & b_1x_1 & ~ & ~ & ~ & ~ & -b_nx_n & ~ & ~\\
  ~ & ~ & ~ & \ddots & ~ & ~ & ~ & ~ & \ddots & ~\\
  ~ & ~ & ~ & ~ & b_{n-1}x_{n-1} & ~ & ~ & ~ & ~ & -b_nx_n\\
\end{pmatrix}
\end{equation}
Here, $a_i = r_i^{-1}s_i$ and $b_i = s_i^{-1}r_i$. Although it is a bit difficult to see from the typesetting, the only nonzero elements are on the main diagonals of each of the four $n(n-1)/2\times n(n-1)/2$ blocks. For example, if $n=2$, the matrix of $\varphi$ is the $2\times 2$ matrix
	\begin{align*}
		\begin{pmatrix}
			-r_2^{-1}s_2 & r_1^{-1}s_1\\
			s_1^{-1}r_1x_1 & -s_2^{-1}r_2x_2
		\end{pmatrix}
	\end{align*}
Then $v(\det(\varphi)) = v(x_1 - x_2)$, and does not depend on $t$. This holds true in general:
\begin{thm}\label{thm:main}
	The determinant $\det(\varphi)$ does not depend on the chosen $t\in X(T,\gamma)$, and
	\begin{align*}
		v(\det(\varphi)) = \sum_{i < j} v(x_i - x_j).
	\end{align*}
	In particular, since $X(T,\gamma)$ is zero-dimensional (Example~\ref{exmp:zerodimtorus}), this is also the dimension of the affine Springer fiber $X(G_0,\gamma)$.
\end{thm}
\begin{proof}
	Let us abuse notation and write $\varphi$ for the matrix of $\varphi$ as in~\eqref{eqn:largematrix}. First, multiply $\varphi$ on the left by the diagonal matrix
	\begin{align*}
		D &= \diag(-b_2b_1x_1,-b_3b_1x_1,\dots,-b_nb_1x_1,-b_3b_2x_2,\dots,-b_nb_2x_2,\dots,-b_nb_{n-1}x_{n-1},1,1,\dots,1)\\
	\end{align*}
In other words, $D$ corresponds to multiplying each of the first $n(n-1)/2$ rows of $\varphi$ by nonzero elements in such a way so that the rows of the upper left $n(n-1)/2\times n(n-1)/2$ block are the same as the rows of the corresponding lower left block. By subtracting the $k$th row from the $[n(n-1)/2 + k]$th row for $k=1,2,\dots,n(n-1)/2$, we get a new matrix that is upper triangular, and whose diagonal entries are those of $\diag(A,B)$, where
\begin{align*}
	A &= \diag(b_1x_1,b_1x_1,\dots,b_2x_2,b_2x_2,\dots,b_2x_2,\dots,b_{n-1}x_{n-1}),\\
	B &= \diag(b_2(x_1-x_2),b_3(x_1-x_3),\dots,b_n(x_{n-1}-x_n)).
\end{align*}
Hence, the determinant of $\varphi$ is $\det(\diag(A,B))/\det(D)$, which is
\begin{align*}
	\det(\varphi) &= \frac{\prod_{i < j} b_ib_jx_i(x_i - x_j)}{\prod_{i < j} -(b_ix_ib_j)}
	= (-1)^{n(n-1)/2} \prod_{i < j} (x_i - x_j).
\end{align*}
Taking the valuation of this gives the desired result.
\end{proof}
\begin{rem}\label{rem:cartansub}By definition, a \emph{Cartan subspace} is a subspace $\afr$ of $\gfr_1$ that is maximal with respect to being commutative and consisting entirely of semisimple elements. One easily calculates that the centraliser of $\gamma$ is of the form
\begin{align*}
    \left\{ (X, X\beta) : X\beta = \beta X \right\}.
\end{align*}
Since $\gamma = \diag(x_1,\dots,x_n)$ is regular, which is the same thing as saying $x_i\not= x_j$ for all $i\not= j$, we see that the centraliser of $\gamma$ is just the commutative subspace
\begin{align*}
    \afr(R) = \left\{ \bigl(\diag(c_1,\dots,c_n), \diag(c_1x_1,\dots,c_nx_n)\bigr) : c_i\in R\right\},
\end{align*}
which consists entirely of semisimple elements, and hence is a Cartan subspace. Since $\gamma' = \bigl( \diag(c_1,\dots,c_n),\diag(c_1x_1,\dots,c_nx_n)\bigr)\in \afr(k)$ is conjugate to $\bigl(I_n, \diag(c_1^2x_1,\dots,c_n^2x_n)\bigr)$, we can also use our formula to compute that $\dim X_{\gamma'} = \sum_{i < j} v(c_i^2x_i - c_j^2x_j)$ whenever $\gamma'$ is regular. This follows since the two affine Springer fibers for two conjugate elements are isomorphic.
\end{rem}
\begin{rem}
One can define a map $X(G_0,\gamma)\to \G_m(F)/\G_m(\ofr)\cong \Z$ by $(A,B)\mapsto v(\det(A^{-1}B))$, which has nonempty fibers over a finite set of points in $\Z$. The fiber over the point $0\in \Z$ is the affine Springer fiber $X(\GL_n,\beta)$, which by Kazhdan and Lusztig's formula also has dimension $\sum_{i < j} v(x_i - x_j)$. Theorem~\ref{thm:main} then says that all the other fibers also have this dimension.\end{rem}
\begin{exmp}\label{exmp:unitary}
    We can use the formula in Theorem~\ref{thm:main} to calculate the dimension of affine Springer fibers for forms as well, since we can always base change with respect to a field extension of $k$ to calculate dimension. For instance, let $k'/k$ be a quadratic extension and $E/F$ the corresponding extension of Laurent series fields. We use the notation $\overline{x}$ to denote the action of the nontrivial element of $\Gal(E/F)$ on $x$. Define the $n\times n$ matrix $J_n$ to be the antidiagonal matrix
\begin{align*}J_n=
	\begin{pmatrix}
		~ & ~ & 1\\
		~ & \diagup & ~\\
		1 & ~ & ~
	\end{pmatrix}.
\end{align*}
We define the $n\times n$ unitary group as the functor on $F$-algebras $R$ given by
\begin{align*}
	\UU_n(R) = \{ g\in \GL_n(R\otimes_F E) : J_n\overline{g}^{-t}J_n = g\}
\end{align*}
For ease of exposition, we just consider the embedding $\UU_{2n}\times\UU_{2n}\hookrightarrow \UU_{4n}$. Let $\theta$ be the involution defined as conjugation by
\begin{align*}
    \begin{pmatrix}
        ~ & ~ & ~ & J_n\\
        ~ & ~ & -J_n & ~\\
        ~ & -J_n & ~ & ~\\
        J_n & ~  & ~ & ~
    \end{pmatrix}.
\end{align*}
One checks that
\begin{align*}
    \afr(R) = \left\{ \diag(x_1,\dots,x_{2n},-x_{2n},\dots,-x_1) : x_i\in R \right\}
\end{align*}
is a maximal commuting subspace of semisimple elements in the $-1$-eigenspace of $\theta$ on the Lie algebra $\ufr_{2n}$ of $\UU_{4n}$. The representation of $\UU_{4n}^\theta\cong \UU_{2n}\times\UU_{2n}$, after base change to $E$, is conjugate to the $\GL_{2n}\times\GL_{2n}\hookrightarrow \GL_{4n}$ representation we have studied in this paper. Letting $\beta = \diag(x_1,\dots,x_{2n})$, the semsimple element $\gamma = \diag(x_1,\dots,x_{2n},-x_{2n},\dots,-x_1)\in\afr(F)$ is conjugate to 
\begin{align*}
    \begin{pmatrix}
        0 & \beta\\
        \beta & 0
    \end{pmatrix},
\end{align*}
which is in turn conjugate to
\begin{align*}
    \begin{pmatrix}
        0 & I_{2n}\\
        \beta^2 & 0
    \end{pmatrix}.
\end{align*}
Hence, by the formula in Theorem~\ref{thm:main},
\begin{align*}
    \dim X(\UU_{2n}\times\UU_{2n},\gamma) = \sum_{i < j}^{2n} [v(x_i - x_j) + v(x_i + x_j)].
\end{align*}
We note that a similar formula also holds for $\UU_n\times\UU_n\hookrightarrow \UU_{2n}$ when $n$ is odd.
\end{exmp}

In general, not all Cartan subspaces will be $G_0(F)$-conjugate. However, it is tempting to make the following conjecture that we hope to resolve in a future paper:
\begin{conj}
	For any regular semismple $\gamma = \left(\begin{smallmatrix} 0 & X\\ Y & 0\end{smallmatrix}\right) \in\gfr_1(F)$,
	\begin{align*}
		\dim_k X(G_0,\gamma) = v(\det([XY,-]).
	\end{align*}
    As before, $XY$ is considered as an element of $\mathfrak{gl}_n$ and $[XY,-]$ is the commutator map restricted to the Lie algebra of the upper triangular Borel subgroup.
\end{conj}

\bibliographystyle{alpha}
\bibliography{/media/jpolak/KINGSTON/math_papers/polakmain.bib}

\end{document}